\documentclass[a4paper,reqno]{amsart}

\usepackage{graphics,graphicx,color} 
\usepackage{amssymb}
\usepackage{srcltx, bm}

\usepackage{amsfonts}
\usepackage{latexsym}
\usepackage{amsmath,amsthm,amssymb, color, amscd,marginnote}
\usepackage[all]{xy}
 \usepackage{mathrsfs}
\usepackage{MnSymbol}
\usepackage{marginnote}
\usepackage{hyperref}
\hypersetup{colorlinks, linkcolor=blue,citecolor=blue}

\usepackage{graphicx,psfrag,epsfig}

\theoremstyle{plain}
\newtheorem{theorem}{Theorem}

\newtheorem{lemma}[theorem]{Lemma}
\newtheorem{corollary}[theorem]{Corollary}

\newtheorem{proposition}[theorem]{Proposition}
\theoremstyle{definition}
\newtheorem{definition}[theorem]{Definition}
\newtheorem{remark}[theorem]{Remark}
\newtheorem{assumption}[theorem]{Assumption}
\newtheorem{example}[theorem]{Example}

\def\eps{\varepsilon}

\def\lan{\langle}
\def\ran{\rangle}

\def\bdef{\begin{definition}}
\def\endef{\end{definition}}
\def\bthm{\begin{theorem}}
\def\ethm{\end{theorem}}
\def\blm{\begin{lemma}}
\def\elm{\end{lemma}}
\def\brm{\begin{remark}}
\def\erm{\end{remark}}
\def\bprop{\begin{proposition}}
\def\eprop{\end{proposition}}
\def\bcor{\begin{corollary}}
\def\ecor{\end{corollary}}
\def\be{\begin{eqnarray}}
\def\ee{\end{eqnarray}}
\def\beal{\begin{aligned}}
\def\enal{\end{aligned}}

\def\om{\omega}

\def\e{\varepsilon}

\def\phi{\varphi}

\def\R{\mathbb R}
\def\C{\mathbb C}
\def\Nn{\mathbb N}
\def\T{\mathbb T}

\def\Z{\mathbb Z}

\def\M{\mathcal M}

\def\cP{\mathcal P}

\def\cF{\mathcal F}
\def\~{\tilde}

\def\cA{\mathcal A}

\def\cD{\mathcal D}

\def\PP{\mathbf{P}}
\def\EE{\mathbf{E}}

\def\p{\partial}
 \newcommand{\strela}{\rightharpoonup}
 
  \def\BR{\bar B_R}

\def\be{\begin{equation}}
\def\ee{\end{equation}}
\def\bdef{\begin{definition}}
\def\endef{\end{definition}}
\def\blm{\begin{lemma}}
\def\elm{\end{lemma}}
\def\beal{\begin{aligned}}
\def\enal{\end{aligned}}

\newtheorem*{Pf}{Proof}
\renewenvironment{proof}{\begin{Pf} \begin{upshape}} {\end{upshape} \qed\end{Pf}}

\numberwithin{equation}{section}
\numberwithin{theorem}{section}

\title[On averaging and mixing for SPDEs]
{On averaging and mixing for stochastic PDEs
}
\begin{document}

\author{Guan Huang}
\address{Yau Mathematical Sciences Center, Tsinghua University, Beijing, China}
\email{huangguan@tsinghua.edu.cn}
\author{Sergei Kuksin}
\address{Institut de Math\'emathiques de Jussieu--Paris Rive Gauche, CNRS, Universit\'e Paris Diderot, UMR 7586, Sorbonne Paris Cit\'e, F-75013, Paris, France \& School of Mathematics, Shandong University, Jinan, Shandong,  China }
\email{sergei.kuksin@imj-prg.fr}
\maketitle

\begin{abstract}
We examine the convergence in the Krylov--Bogolyubov averaging for nonlinear stochastic perturbations of linear PDEs with pure imaginary spectrum and 
show that if the involved effective equation is mixing, then the convergence is uniform in time.
\end{abstract}

\section{Introduction} 
The Krylov--Bogolyubov averaging for stochastic PDEs  which we are concerned with in this work, means the following. Starting with a linear PDE on a 
torus (or on a bounded domain) with pure imaginary spectrum we consider its $\e$-small nonlinear stochastic perturbation. Then the above mentioned 
averaging describes the behaviour of the distributions of actions of solutions for the perturbed equation on time-intervals of order $\e^{-1}$. Here the actions 
of solutions are made by  the halves of squared norms of their Fourier coefficients with respect to the basis, made by eigenfunctions of the original linear system.
 Description of the limit is made via an auxiliary {\it effective equation} which is another nonlinear 
 stochastic equation whose 
nonlinearity is made from resonant terms of the nonlinear part of the perturbation. The effective equation may be mixing, and then as time goes to infinity
its solutions converge in distributions to a statistical  equilibrium, given by a measure in a function space. An observation which we make in this work is that 
in the mixing case the convergence in distribution of actions of solutions for the perturbed equation to those of solution for the effective equation, 
described by the Krylov--Bogolyubov averaging, is uniform in time. 

The Krylov--Bogolyubov  
 averaging under discussion applies to various classes of stochastic  PDEs, depending on the type of the original unperturbed  linear system. In 
Sections~\ref{s_2}-\ref{s_3} we discuss in details the averaging for stochastic complex Ginzburg--Landau (CGL) equations, regarded as perturbations of 
linear Schr\"odinger equations, and in Section~\ref{s_nlw} briefly repeat the argument for stochastic nonlinear wave equations.

\medskip
\noindent {\it Notation}. For a Banach space $B$ and $R>0$ we denote 
$
\bar B_R(B) = \{ b\in B: | b|_B \le R\};
$
for a metric space $M$, $\cP(M)$ stands for the space of probability Borel measures on~$M$. By $\strela$ we denote the weak convergence of measures and by $\cD(\xi)$ we denote the distribution of a random variable $\xi$. For a function $f$ and a measure $\mu$ we denote 
$
\lan f, \mu\ran= \int\! f\,d\mu.
$

\section{CGL: the setting and  result}\label{s_2}
We  consider a stochastic  CGL equation on a torus $T^{ D}:=\mathbb{R}/(L_1\mathbb{Z})\times\mathbb{R}/(L_2\mathbb{Z})\times\cdots\times\mathbb{R}/(L_{ D}\mathbb{Z})$, $L_1,\dots, L_{ D}>0$,
\begin{equation}\label{cgl-1}
u_t+i(-\Delta+V(x))u=\varepsilon\mu\Delta u+\varepsilon\mathscr{P}(\nabla u,u)+\sqrt{\varepsilon}\eta(t,x), \; \;\;u=u(t,x), \; x\in T^{ D} ,
\end{equation}
where $\mu \in \{0,1\}$,  $\mathscr{P}:\mathbb{C}^{{ D}+1}\to\mathbb{C}$ is a $C^\infty$-smooth function,  
$\varepsilon\in(0,1]$ is a small parameter, 
the random force $\eta(t,x)$ is white in time and regular in $x$, and the potential $V(x)$  is a real smooth function.
 If $\mu=0$, the nonlinearity $\mathscr{P}(\nabla u,u)$ should not  depend on $\nabla u$.  
For simplicity we assume that $\mu=1$ (the case $\mu=0$ can be treated similarly). Again only to simplify  presentation we also assume that 
$V(x)>0$ for all  $x$.

For any $s\in\mathbb{R}$ we denote by $H^s$ the Sobolev space of complex functions on $T^D$,  provided with the norm $\|\cdot\|_s$, 
$$\|u\|_s^2=\langle(-\Delta)^su,u\rangle+\langle u,u\rangle,\; \text{if} \; s\geqslant0,$$
where $\langle\cdot,\cdot\rangle$ is the real scalar product in $L^2(T^D;\C)$,
\[ 
\langle u,v\rangle=\Re\int_{T^d}u\bar v\, {\rm{d}}x,\quad  u,v\in L^2(T^D;\C).
\]
  Let $\{\mathbf{e}_{l}(x),l\in\mathbb{N}\}$ be the usual trigonometric basis of the space $L^2(T^D)$, parametrized by natural numbers. 
Then $-\Delta \mathbf{e}_l = \kappa_l \mathbf{e}_l$,  $\kappa_l\ge0$, and we assume that $0=\kappa_1<\kappa_2\le\kappa_3,\dots$. 
  We take the force term $\eta(t,x)$  in \eqref{cgl-1}   to be  of the form
  \be\label{eta}
  \eta(t,x)=\frac{\partial }{\partial t}\xi(t,x),\quad \xi(t,x):=\sum_{l\geqslant1}b_l\beta_l(t)\mathbf{e}_l(x).
  \ee
  Here $\beta_l(t)=\beta_l^R(t)+i\beta_l^I(t)$, where $\beta_l^R(t)$, $\beta_l^I(t)$, $l\geqslant1$, are independent real-valued standard Brownian motions, defined on a complete probability space $(\Omega, \mathcal{F},\mathbf{P})$ with a filtration $\{\mathcal{F}_t;t\geqslant0\}$.\footnote{So $\{\beta_l(t)\}$
  are standard independent
 complex Brownian motions.  }
  As a function of $x$,  $\xi(t,x)$ is assumed to be smooth in the sense that the real numbers $b_l, l\ge1$, decays to zero faster than any 
  negative degree of $l$.
  
  Introducing the slow time $\tau=\varepsilon t$, we rewrite  eq. \eqref{cgl-1} as
  \begin{equation}\label{m-equation1}
  \dot{u}+\varepsilon^{-1}(-\Delta+V(x))u= \Delta u+\mathscr{P}(\nabla u,u)+\dot\xi(\tau,x),\quad u=u(\tau,x), \;x\in T^D,
  \end{equation}
  where   the upper dot  stands for ${\p}/{\p\tau}$, and $\xi(\tau,x)$ is as in \eqref{eta} with $t:=\tau$ and with another set of standard independent
 complex Brownian motions $\beta_l$. 
  Here and below we write stochastic PDEs with additive noise as 
  nonlinear PDEs with forcing terms of the form \eqref{eta}. 
  \begin{definition}
  
 If $L$ and $E$ are Banach spaces with norms $|\cdot|_L$ and $|\cdot|_E$,
 then $\text{Lip}_m(L, E)$, $m\ge0$, is the collection of maps $F:L \to E$ such
 that for any $R\geqslant1$,
\[
\sup_{R>0} \Big( (1+ |R|)^{-m}\big(\text{Lip}(F|_{\BR(L)})+\sup_{v\in \BR(L)}|F(v)|_E\big)\Big)  <\infty,
\]
where 
 $\text{Lip}(f)$ is the Lipschitz constant of a mapping~$f$.  \end{definition}
 
 We 
  make the following assumption concerning  the well-posedness of eq. \eqref{m-equation1}. There and everywhere below in our paper
\[
\begin{split}
\text{\it either always the indices $s$ of involved Sobolev spaces $H^s$ are integer,}\\
\text{\it  or always they are any real numbers.} 
\end{split}
\]

  \begin{assumption}\label{assume1} There exist numbers  $0<s_1<s_2<+\infty$ and $\bar m\in\mathbb{N}$ such that for  each $s\in (s_1,s_2)$,
  \begin{enumerate}      \item  the mapping $H^s\to H^{s-1}$: $u\mapsto \mathscr{P}(\nabla u,u)$ belongs to $\text{Lip}_{\bar m}(H^{s},H^{s-1})$;
  \item for any $ \e\in (0,1]$ and 
   $u_0\in H^s$ equation  \eqref{m-equation1} has a unique strong solution $u^\om(\tau;u_0)$, equal $u_0$ at $ \tau=0$, 
  defined for $\tau\geqslant0$, and 
 $$ 
  \mathbf{E}\sup_{\theta\leqslant\tau\leqslant\theta+1}\|u(\tau;u_0)\|_s^{ 2\bar m}\leqslant C_s(\|u_0\|_s),\quad  \forall \theta\geqslant0,
$$ 
  where $C_s(\cdot)$ is a continuous and non-decreasing function. 
   \item
   If we work in the category of Sobolev spaces with integer indices, then the integer segment $(s_1, s_2)\cap \Z$ contains at least two points. 
   \end{enumerate}
  \end{assumption}
  
   Under the above assumptions  the  family of solutions $\{u^\omega(\tau; u_0), u_0\in H^s\}$ defines in the spaces $H^s$, $s\in(s_1, s_2)$, 
     Markov processes.

  Assumption \ref{assume1} is satisfied for many nonlinearities $\mathscr P$. In particular it holds if 
  \be\label{nonlin}
  \mathscr P(\nabla u, u) = -u + {\frak z} f_p(|u|^2)u, \quad {\frak z} \in \C, \ |{\frak z}|=1, \Im {\frak z} \le 0,\ \Re{\frak z} \le0 ,
  \ee
  where $f_p(r)$ is a non-decreasing smooth function on $\R$, equal $r^p$ for $r\ge1$. The degree $p\ge0$ is any if $D=1,2$, and 
  $p< 2/(D-2)$ if $D\ge3$. See \cite[Section~5]{HKM}.

 We denote by $A_V$ 
 the Schr\"odinger operator  
 $$A_Vu:=-\Delta u+V(x)u.$$
 Let $\{\lambda_l\}_{l\geqslant1}$ be its eigenvalues, ordered in such a way that 
 \[
 0<\lambda_1\leqslant\lambda_2\leqslant\lambda_3\leqslant\cdots \,
 \]
 (we recall that $V>0$)  and let 
$\{\varphi_l,\;l\geqslant1\}\subset L^2({T^D})$ be an orthonormal basis, 
formed by the corresponding 
 eigenfunctions. We denote $\Lambda=(\lambda_1,\lambda_2,\dots)$ and call 
 $\Lambda$ the {\it frequency vector} of eq.~\eqref{m-equation1}.
  For a complex-valued function $u\in H^s$ we denote by 
 $$
  \Psi(u):=v=(v_1,v_2,\dots),\quad v_k\in\mathbb{C},
  $$ 
  the vector of its Fourier coefficients with respect to the basis $\{\varphi_l\}_{l\geqslant1}$: $u=\sum_{l\geqslant1}v_l\varphi_l$.
  Note that $\Psi$ is a real operator: it maps real functions $u(\cdot)$ to real vectors $v$. 
  In the space of complex sequences $v=(v_1,v_2,\dots)$ we introduce the norms 
  \[\left|v\right|^2_s=\sum_{k=1}^{\infty}\left(|\lambda_k|^s+1\right)|v_k|^2,\quad
  s\in\mathbb{R}\,,
  \] 
  and denote $h^s=\{v:|v|_s<\infty\}$. Then  $\Psi$ defines   an     isomorphism 
   between the spaces $H^s$ and $h^s$, for any $s$. 
  
  Now we write Eq.\ (\ref{m-equation1}) in the $v$-variables:
  \begin{equation}\label{eq-v1}
  \dot{v}_k+\varepsilon^{-1}i\lambda_kv_k=- \lambda_kv_k+P_k(v)+\sum_{l=1}^\infty\Psi_{kl}b_l\dot\beta_l,\quad k\in\mathbb{N}.\end{equation}
  Here the $k$-th equation is obtained as the $L^2$-scalar product with $\phi_k$ of eq.~\eqref{m-equation1}, 
  where $u= \Psi^{-1}v$. So 
   $$
   P(v)=(P_k(v),\;k\in\mathbb{N})=\Psi\Big(  V(x)u+\mathscr{P}(\nabla u, u)\Big),\quad u=\Psi^{-1}v,
   $$
  and    $\Psi_{kl} =\langle \phi_k, \mathbf{e}_l \rangle$ (thus $\big( \Psi_{kl}  \big)$ is the matrix of the operator $\Psi$ 
  with respect to the trigonometric basis in $L^2(T^D)$ and the natural basis in $h^0$).

  Our  task is to study the dynamics of eq. \eqref{eq-v1} when $\varepsilon\ll 1$. 
  An efficient way to deal with this problem is through the  {\it interaction representation}, which means transition from  variables $\{ v_k(\tau)\}$ 
  to variables  $\{ a_k(\tau)\}$, where 
$$
a_k(\tau)=e^{i\varepsilon^{-1}\lambda_k\tau}v_k(\tau), \quad k\ge1.
$$
In the $a$-variables equations \eqref{eq-v1} read 
  \begin{equation}\label{eq-a1}
  \dot{a}_k(\tau)=- \lambda_ka_k+e^{i\varepsilon^{-1}\lambda_k\tau}P_k(\Phi_{-\varepsilon^{-1}\Lambda\tau}a)+
  e^{i\varepsilon^{-1}\lambda_k\tau}\sum_{l=1}^\infty\Psi_{kl}b_l\dot\beta_l,\quad k\in\mathbb{N}\,,\;\; \tau\ge0, 
  \end{equation}
  where for a vector  $\theta=(\theta_k,\;k\in\mathbb{N})\in\mathbb{R}^{\infty}$, $\Phi_{\theta}$ stands for the rotation 
  in $h^s$,  defined by 
 \begin{equation}\label{Phi-d}
  \Phi_{\theta}v=v',\quad v'_k=e^{i\theta_k}v_k\quad\forall\,k\,.
\end{equation}
  Clearly operators $\Phi_\theta$ define isometries  of all   spaces $h^s$.  By \mbox{$a^\e(\tau; v_0) = \big( a_k^\e(\tau; v_0) , k\ge1)$} 
  we denote a solution $u(\tau;u_0)$, written in the $a$-variables. It solves system \eqref{eq-a1} with the initial data
  $$
  a(0) =v_0 := \Psi(u_0),
  $$ 
  and in view of Assumption \ref{assume1}.(2), 
  \be\label{assume-bound11}
  \mathbf{E}\sup_{\theta\leqslant\tau\leqslant\theta+1}|a^\e(\tau;v_0)|_s^{2\bar m}
  \leqslant C'_s(|v_0 |_s),\quad  \forall \theta\geqslant0.
  \ee

  In order to describe  the dynamics of eq. \eqref{eq-a1} with $\varepsilon\ll1$ 
  we introduce an {\it  effective equation}: 
  \begin{equation}\label{eq-ef1}
  \dot{ a }_k=- \lambda_k a _k+R_k( a )+\sum_{l=1}^\infty B_{kl}\dot\beta_l,\quad k\in\mathbb{N},\;\;\; \tau\ge0.
  \end{equation}
  Here $\{B_{kl}, k,l\geqslant1\}$ is the principal square root of  infinite   matrix $\{A_{kl},\; k,l\geqslant1\}$,\,\footnote{
  The matrix $(A_{kl})$ defines a non-negative
   compact   self-adjoint operator in the space $l^2$. So its principal square root (which defines  another  non-negative
   compact   self-adjoint operator) exists.
  }
  $$A_{kl}=\begin{cases}\sum_{j\geqslant1}b_j^2\Psi_{kj}\Psi_{lj},&\text{if }\lambda_k=\lambda_l,\\
  0,&\text{else},\end{cases}$$
  and
  \begin{equation}\label{limit-eq}
  R( a ):=(R_k( a ),\;k\in\mathbb{N})=\lim_{T\to\infty}\frac{1}{T}\int_0^{T}\Phi_{\Lambda t}P(\Phi_{-\Lambda t} a )dt
   \end{equation}
  (see  \cite[Lemma 4]{HKM} concerning this limit). 
  Under (1) of Assumption \ref{assume1}, for $s\in(s_1,s_2)$, the mapping $h^s\to h^{s-1}: v\mapsto R(v)$ belongs to $\text{Lip}_{\bar m}(h^s, h^{s-1})$, see  \cite{HKM}. Therefore for any $a_0\in H^s$ with  $s\in (s_1,s_2)$, a strong solution $a^0(\tau; a_0)$ of eq. \eqref{eq-ef1}, equal $a_0$ at $\tau=0$, is unique and exists at least locally. 
  In  \cite[Theorem~2 and Proposition~1]{HKM} we proved the following result:

  \begin{theorem}\label{t_0} 
  If Assumption \ref{assume1} holds, then for any $s_1< s_* <\bar s <s_2$ and any
  $v_0\in h^{\bar s}$ we have: 
  
  i) eq. \eqref{eq-ef1} has a unique strong solution $a^0(\tau; v_0)$, $\tau \ge0$, equal to $v_0$ at $\tau=0$. It belongs to
  $
  C([0, \infty), h^{\bar s})
  $
  a.s., and for any $\theta\ge0$, 
  \begin{equation}\label{up-1}
 \mathbf{E}\sup_{\tau\in[\theta,\theta+1]}|a^0(\tau;v_0)|_{\bar s}^{2\bar m}\leqslant C'_{\bar s}(|v_0|_{\bar s}). 
\end{equation}

ii) For any $T>0$,
  \be\label{conv}
  \cD \big(a^\e (\tau; v_0)\! \mid_{[0,T]} \!\big)  \strela \cD \big(a^0 (\tau; v_0)\! \mid_{[0,T]} \!\big) \quad \text{in}\;\; \cP\big(C([0,T]); h^{s_*}) \big)\;\;
  \text{as} \; \;\e\to0. 
  \ee
  \end{theorem}   
  
    \begin{remark}\label{r_2.4}
       1) In \cite{HKM} it was proved only the second assertion of the theorem, implying that 
     $a^0$ is a solution of eq.~\eqref{eq-ef1}   in $h^{s_*}$. Then
      \eqref{assume-bound11} and   Fatou's  lemma imply \eqref{up-1}. Indeed, for $M\in \Nn$ let $\Pi_M(a_1, a_2, \dots) = 
    (a_1, \dots, a_M, 0, \dots)$. Then by \eqref{assume-bound11}     and convergence  \eqref{conv}, for each $\theta\ge0$ and any
    $M,N \in \Nn$
    $$
    \EE \sup_{\tau\in[\theta,\theta+1]}N\wedge |\Pi_M a^0(\tau;v_0)|_{\bar s}^{2\bar m}\leqslant C'_{\bar s}(|v_0|_{\bar s}). 
    $$
    Sending first $N\to\infty$ and then $M\to\infty$ and using Fatou's lemma we recover \eqref{up-1}.
     Since $a^0$ is a solution in $h^{s_*}$, then dues to  \eqref{up-1} it is a solution in $h^{\bar s}$. Uniqueness of a solution is obvious. 
    
    2) It follows immediately that  \eqref{conv} also holds for a solution $a^\e(\tau; v_0)$ if the initial data      $v_0$ is a r.v. in $h^{\bar s}$, independent from the 
    random field $\xi$. Moreover, a simple analysis of the proof in \cite{HKM} implies that if $|v_0^\om|_{\bar s} \le M$ a.s., then the rate of  convergence 
     \eqref{conv}     depends only on $M$ (and, of course, on      $s_*<\bar s$).

     3) Theorem's assertion with the same proof remains true if in the r.h.s. of eq.~\eqref{m-equation1} we replace the viscosity $\Delta u$ by the 
     hyperviscosity $-(-\Delta+1)^r u $, $r\in \Nn$ 
      (provided that Assumption \ref{assume1} holds for the equation). The equations with    hyperviscosity are important for some applications, see below Example~\ref{e_3}.

      4)  In \cite{HKM} the bounds  $ \mathbf{E}\sup_{\theta\leqslant\tau\leqslant\theta+1}\|u(\tau;u_0)\|_s^{m}$ are assumed for all  $m$-th moments of solutions,        $m\in\mathbb{N}$. 
      It was done only for simplicity, and 
    under  the additional restriction in Assumption \ref{assume1}.(1), which specifies the growth of the the nonlinearity $\mathscr{P}$, only the bounds for the 
      moments with $m\le 2\bar m$ are needed for the proof. 
      This difference is rather insignificant since  using the standard techniques of exponential supermartingales  it is usually 
      easy to obtain bounds for all  $m$-th moments after the  second moments are estimated. 
  \end{remark}

  Again, solutions  $a^0(\tau; v_0)$ of \eqref{eq-ef1} define in the space    $h^{s_*}$ a  Markov process. 
  
  Our goal in this work is to prove that if effective 
  equation \eqref{eq-ef1} is mixing, then convergence \eqref{conv} is such that
   $ \cD (a^\e (\tau; v_0) )\strela  \cD (a^0 (\tau; v_0))$ uniformly in $\tau\ge0$, with respect to the dual-Lipschitz distance. 
   We recall 
   
  \begin{definition}\label{def4}
   Let $M$ be a complete and separable metric space. For any two measures $\mu_1$, $\mu_2\in\mathcal{P}(M)$ 
   the dual-Lipschitz distance between them is 
  $$
  \|\mu_1-\mu_2\|_{L,M}^*:=\sup_{f\in C(M), \, |f|_{L,M}\leqslant1}\big|\langle f,\mu_1\rangle-\langle f,\mu_2\rangle\big|\leqslant2,
  $$
  where $|f|_{L,M}=\text{Lip} f+\|f\|_{C(M)}$. 
  \end{definition}
  
  For future usage we note that for $s \ge s'$  space $\cP(h^{s})$ is naturally embedded in $\cP(h^{s'})$, and it easily follows from the definition that 
  for $\mu_1, \mu_2 \in \cP(h^{s})$ the distance $\|\mu_1-\mu_2\|_{L, h^{s'}}^*$ is a non-decreasing function of $s'\le s$. 
  \smallskip
  
  To proceed we have to 
  assume that the effective equation is mixing, and the  rate of mixing  is uniform for initial data from bounded sets: 
  
  \begin{assumption}\label{assume2}   For some $s_*\in(s_1,s_2)$  effective equation \eqref{eq-ef1} is mixing in the space $h^{s_*}$ with a stationary measure
   $\mu^0\in \mathcal{P}(h^{s_*})$,  and for each $M>0$ and $v\in \bar{B}_M(h^{s_*})$, we have 
  \begin{equation}\label{mix-b1}
  \|\cD(a^0(\tau;v))-\mu^0\|_{L,h^{s_*}}^*\leqslant  \frak g_M(\tau),
  \end{equation}
  where $\frak g$ is a continuous function of $(M,\tau)$ which goes to zero when $\tau\to\infty$. 
  \end{assumption}
  
  Relation \eqref{mix-b1} is a mild specification of the mixing in eq. \eqref{eq-ef1}, and a proof of the  latter in fact usually establishes the former.

  Our main result is the following: 
  
  \begin{theorem}\label{thm-uniform} Under Assumptions \ref{assume1} and \ref{assume2}, for any  $\bar s\in(s_*,s_2)$  and any $v_0\in h^{\bar s}$ 
 
  $$
  \lim_{\varepsilon\to0}\sup_{\tau\geqslant0}\|\mathcal{D}(a^{\varepsilon}(\tau;v_0))-\mathcal{D}(a^0(\tau;v_0))\|_{L,h^{s_*}}^*=0,
  $$
  where $a^{\varepsilon}(\tau;v_0)$ and $a^0(\tau;v_0)$ solve respectively equations  \eqref{eq-a1} and \eqref{eq-ef1} with initial conditions
   $a^{\varepsilon}(0;v_0)=a^0(0;v_0)=v_0$. Moreover,  for any $M>0$ the above convergence is uniform for $v_0\in \bar B_M(h^{\bar s})$. 
  \end{theorem}
  
For $v=(v_1,v_2,\dots)\in h^s$ we  introduce 
the vector of action variables $I(v)=(I_1(v),I_2(v),\dots )$, where $I_k(v)=\frac{1}{2}|v_k|^2$, $k=1,2,\dots$. 
Then $I(v) \in  h_I^s \cap  \mathbf{R}_+^\infty$, where  $h^s_I$ is the weighted $l^1$-space  with the norm 
$
 |I|_{I,s}=\sum_{k=1}^\infty(|\lambda_k|^s+1)|I_k|.
 $
Since the interaction representation does not change the actions, then for  the action variables of  solutions for the original  equations
 \eqref{eq-v1}  we have 
\begin{corollary}\label{coroll-cgl}
Under the assumptions of Theorem \ref{thm-uniform}  actions of a  solution $v^\e(\tau; v_0)$ for eq.~\eqref{eq-v1} with $v_0\in h^{\bar s}$, $\bar s>s_*$
 satisfy
 \[
  \lim_{\eps\to0}\,
  \sup_{\tau\geqslant0}\|\cD\big(I( v^{\eps}(\tau;v_0)) \big) -\cD\big((I( a^{0}(\tau;v_0))\big) \|_{L,h_I^{s_*}}^* =0.
  \]
\end{corollary}
%

\subsection{Examples} 

For a finite-dimensional stochastic equation 
  \be\label{fin_dim}
  \dot a_k(\tau) = e^{i\e^{-1} \lambda_k\tau} P_k(\Phi_{-\e^{-1} \Lambda\tau} a) +  e^{i\e^{-1} \lambda_k\tau}\sum_{l=1}^N \M_{kl} \dot\beta_l (\tau), \quad
  k=1, \dots, N, 
  \ee
  where $a(\tau) =(a_1,\dots, a_N)(\tau) \in \C^N$ and $\Lambda=( \lambda_1, \dots, \lambda_N)$, satisfying a natural analogy of Assumptions~\ref{assume1},~\ref{assume2},   a natural version of Theorem~\ref{thm-uniform} holds. 
  If the matrix $(\M_{kl})$ is non-degenerate and Assumption~\ref{assume1} is met,  then a convenient sufficient condition for 
  Assumption~\ref{assume2}  follows from the Khasminski criterion for mixing (details will be given in paper on the finite-dimensional stochastic averaging
  under preparation).\footnote{ In \cite[Theorem~2.9]{AD} the uniform in time convergence as above is proved for a class of systems \eqref{fin_dim}. The proof
  in \cite{AD} is based on the observation that  the corresponding  $v$-equations are mixing with a rate of mixing, independent from $\e$. }
   For infinite-dimensional stochastic systems an instrumental criterion of  mixing is not known yet. We believe that as in finite dimensions,
  Assumption~\ref{assume2} holds for effective equations for various equations  \eqref{m-equation1}, satisfying Assumption~\ref{assume1}, at least if 
  the random forces   \eqref{eta}  is such that 
  all $b_l$'s are non-zero. Examples below are given to support this belief.

  \begin{example}
  Consider eq. \eqref{m-equation1}, where $\mathscr P(\nabla u, u) =\mathscr P(u)$ has the form \eqref{nonlin} with $\frak z=-1$. Since the function $f_p$ is monotone, 
  then mapping  $\mathscr P$ also is monotone in the sense that 
  $$
  \langle \mathscr P(u) -\mathscr P(v), u-v \rangle \ge 0 \quad \forall u,v \in L^{\infty}( T^D; \C).
  $$
  From definition \eqref{limit-eq} we see that operator $R$ is monotone as well. Taking any two vectors $v^1, v^2 \in h^s$ and denoting 
  $
  w(\tau) = a^0(\tau; v^1)- a^0(\tau; v^2)
  $
  we immediately derive from the equation for $w$ that 
  $
  (d/dt) |w^\om(\tau)|_0^2 \le -\tfrac12 \lambda_1  |w^\om(\tau)|_0^2.  
  $
  So 
  $$
   |w^\om(\tau)|_0 \le e^{-\lambda_1 \tau} | v^1-v^2|_0 \quad \forall\, \tau\ge0,\ \forall\om. 
  $$
  This very strong a-priori  estimate implies Assumption \ref{assume2} after some non-complicated work. 
  \end{example}
  
  \begin{example}\label{e_2}
  Now let  $\mathscr P(\nabla u, u) =\mathscr P(u)$ has the form \eqref{nonlin}, where  $\Re\frak z<0$ and $\Im\frak z<0$. Then relation 
  \eqref{mix-b1} with 
  $
  {\frak g}_M(\tau) = C(M) e^{-\kappa\tau},
  $
  $\kappa>0$,  and with sufficiently large integer $s_*$
  follows from the abstract theorems in \cite[Theorem~2.1]{Oda} and \cite[Theorem~3.1.7]{KS} since solutions $a^0$ inherit 
  estimates for solutions $a^\e$ via convergence \eqref{conv}. We leave to the reader details of this derivation (which are not quite trivial). 
  \end{example}

   \begin{example}\label{e_3}
   Consider eq. \eqref{m-equation1} with  $\mathscr P(\nabla u, u) =  -i \rho |u|^2 u$, with $V=0$ and with a 
   hyperviscosity instead of viscosity:
   \be\label{NLS}
   \dot u -i \e^{-1} \Delta u = -(-\Delta +1)^ru -i\rho |u|^2 u + \dot\xi(\tau, x),  \quad x \in \T^D_L =\R^D/(L \Z^D).
   \ee
   Here $r\in\Nn$, $\rho>0$ is a scaling factor and $\xi$ has the form $\eqref{eta}$, where all $b_l$'s are non-zero. Now
   $A_V= -\Delta$, so its eigenfunctions are  exponents $e^{i2\pi s\cdot x}$, $s \in L^{-1} \Z^D$. The action variables $I_s(\tau)$ of solutions 
   $u(\tau,x )$ are naturally parametrised by $s \in L^{-1} \Z^D$. The theory of wave turbulence (WT), among other things, examine the behaviour of 
   the expectations of the     actions $\EE I_s(\tau)$ under the limit of WT:
   $$
  \e\to0, \quad L\to\infty, 
   $$
  when $\rho$ is properly scaled with $\e$ and $L$. See \cite{Naz} and 
   the introduction in \cite{DK}. For any dimension $D$ solutions of eq.~\eqref{NLS}
  with $r$ sufficiently large in terms of $D$ satisfy Assumption~\ref{assume1} by the same straightforward proof as for eq.~\eqref{m-equation1} with
  $\mathscr P$ as in \eqref{nonlin}, and the equation also meets Assumption~\ref{assume2}  with sufficiently large integer $s_*$ 
   in view of the abstract theorems, mentioned in Example~\ref{e_2}. So
  Corollary~\ref{coroll-cgl} applies to eq.~\eqref{NLS} with $\rho$ and $L$ fixed, under the limit $\e\to0$ and imply that the expectations of the 
  actions of solutions for 
  \eqref{NLS} converge to those of solutions for  the corresponding effective equation, uniformly in time. The limit $\e\to0$ in eq.~\eqref{NLS}  and in
  similar equations is known in nonlinear physics as the limit of {\it discrete turbulence}. See  \cite{Naz}, \cite{KM}  and 
   \cite[Sections~1.2,\,12.1]{DK}. 
     \end{example}

   \section{Proof of Theorem \ref{thm-uniform}} \label{s_3}
   Below we always assume Assumptions \ref{assume1} and \ref{assume2}. We fix $\bar s>s_*$ and $v_0 \in h^{\bar s}$ as in Theorem~\ref{thm-uniform}.
   We abbreviate  $\|\cdot\|_{L,h^{s_*}}^*$ to $\|\cdot\|_L^*$ and  $a^{\epsilon}(\tau;v_0)$ to
    $a^\epsilon(\tau)$. 
For any $T'\ge0$ 
 we denote by $a_{T'}^0(\tau)$ a weak solution of  the effective equation \eqref{eq-ef1} such that 
$$
\cD (a^0_{T'}(0) )= \cD( a^\eps(T')).
$$  
Note that solution  $a^0_{T'}(\tau)$ depends on $\eps$ and  that  $a^0_0(\tau) = a^0(\tau;v_0)$. 

 The following lemma  follows from Theorem \ref{t_0},   Assumption \ref{assume1}.(2)  and Remark~\ref{r_2.4}.2). 
 
\begin{lemma}\label{average-lm} For  any $\delta>0$ and $T>0$   there exists $\eps_1 =\eps_1(\delta,T)>0$ such that if $\eps\le \eps_1$, then
\be\label{gr}
\sup_{\tau\in[0,T]}\|\cD(a^{\eps}(T'+\tau)) - \cD(a_{T'}^{0}(\tau))\|_{L,h^{s_* }}^*\leqslant \delta/2 , \quad \forall\, T'\geqslant0.
\ee
\end{lemma}

Since by Assumption \ref{assume2}, $\mathcal{D}(a^0(\tau;0))\to \mu^0$ in $\mathcal{P}(h^{s_*})$ as $\tau\to\infty$, 
 then from  estimate \eqref{up-1} and Fatou's lemma we derive  that 
\begin{equation}\label{mu0-b1}
\langle|v|_{\bar s}^{2\bar m},\mu^0\rangle\leqslant  C'_{\bar s}(0):=C_0.
\end{equation}

We need the following statement for the mixing in eq. \eqref{eq-ef1}.
  \begin{lemma}\label{p_suff_cond} 
  For   any solution $a^0(\tau;\mu)\in h^{s*} $, $\tau\geqslant0$, of  effective equation \eqref{eq-ef1} such that $\mathcal{D}(a^0(0;\mu))=:\mu$, where  
  the measure $\mu$ satisfies 
  $$
  \langle |a|_{s_*}^{2\bar m},\mu(da)\rangle=\mathbf{E}|a^0(0;\mu)|_{s_*}^{2\bar m}\leqslant M\quad \text{ for some $M>0$, }
  $$
  we have 
  \be\label{mix-b2}
  \| \cD\big( a^0(\tau ; \mu) \big)- \mu^0\|_{L}^*  \le g_M(\tau, d) \quad \text{if} \;\; \|\mu-\mu^0\|_{L}^* \le d \le2.
  \ee
  Here the function $g:\mathbb{R}_+^3\to\mathbb{R}_+$, $(M,\tau, d)\to g_M(\tau, d)$ is continuous, vanishes with  $d$, converges to zero when $\tau\to\infty$ and is such that for each fixed $M\geqslant0$ the function $(\tau,d)\to g_M(\tau,d)$ is uniformly continuous in $d$ for $(\tau,d)\in[0,\infty)\times[0,2]$. 
     \end{lemma}
 A proof of the lemma      is rather  straightforward but lengthy. It is   given  below in Subsection~\ref{ss_proof_prop}. 
 \smallskip
  
  We  denote  
 $$
 M_*:=C'_{\bar s}(|v_0|_{\bar s} ), 
 $$
 where $C'_{\bar s}(|v_0|_{\bar s})$ is as  in \eqref{assume-bound11} and   \eqref{up-1}. Constants in the estimates below depend on $M_*$, but this dependence usually is not indicated.

\begin{lemma}\label{random-initial-average} 
Take  any  $\delta>0$ and 
choose a $T^* = T^*(\delta)>0$, satisfying   
$$
g_{M_*}( T, 2) \le \delta/4 \qquad \forall\, T\ge T^*.
$$
Then there exists 
$\eps_2 =\eps_2(\delta)>0$ such that if $\eps \le \eps_2$ and  $\|\cD(a^{\eps}(T'))-\mu^0\|_{L}^*\leqslant\delta$ for some $T'\ge0$,   then 
\begin{equation}\label{83a}
\|\cD(a^{\eps}(T'+T^*))-\mu^0\|_{L}^*\leqslant  {\delta}, 
\end{equation}
and 
\begin{equation}\label{83b}
\sup_{\tau\in[T',T'+T^*]}\| \cD(a^{\eps}(\tau))-\mu^0\|_{L}^*\leqslant \tfrac{\delta}2 + \sup_{\tau\ge0} 
g_{M_*}(\tau, \delta ) . 
\end{equation}
\end{lemma}

\begin{proof} 
Let us  choose  $\eps_2(\delta)= \eps_1(   \frac{\delta}{2},T^*(\delta))$, where $\eps_1(\cdot)$ is 
 as in Lemma \ref{average-lm}. Then  we get from \eqref{gr}, \eqref{up-1}, \eqref{mix-b2} with $M=M_*$ 
and the definition of $T^*$   that  for  $\eps\le\eps_2$, 
\[ \begin{split}
\|\cD(a^{\eps}(T'+T^*))-\mu^0\|_{L}^*\leqslant\|\cD(a^{\eps}(T'+T^*))
&-\cD(a^{0}_{T'}(T^*))\|_{L}^* \\
&+\|\cD(a^{0}_{T'}(T^*))-\mu^0\|_{L}^*\leqslant {\delta}.
\end{split}
\]
This proves \eqref{83a}.  Next, in view of  \eqref{gr} 
 and \eqref{mix-b2}, 
\[\begin{split}&\sup_{\theta\in[0,T^*]}\|\cD(a^{\eps}(T'+\theta))-\mu^0\|_{L}^*\\
&\leqslant\sup_{\theta\in[0,T^*]}\|\cD(a^{\eps}(T'+\theta))-\cD(a_{T'}^{0}(\theta))\|_{L}^*+\sup_{\theta\in[0,T^*]}\|\cD(a_{T'}^{0}(\theta))-\mu^0\|_{L}^*\\
&\leqslant\frac{\delta}{2}+\max_{\theta \in[0,T^*] }g_{M_*}(\theta,  \delta). 
 \end{split}\]
This implies \eqref{83b}. 
\end{proof}

We are now ready to prove Theorem \ref{thm-uniform}. 
\begin{proof}[of Theorem \ref{thm-uniform}.]
Let us fix arbitrary $\delta>0$ and take some $0<\delta_1\le \delta/4$. Below in the proof the functions $\eps_1$, $\eps_2$ and $T^*$ are as in 
Lemmas \ref{average-lm} and \ref{random-initial-average}.

i) By the definition of $T^*=T^*(\delta_1)$, \eqref{up-1}  and \eqref{mix-b2}, 
\be\label{84}
\| \cD\big( a^0_{T'} (\tau)) - \mu^0\|_{L}^* \le  \delta_1/4 \quad \forall\, \tau\ge T^*,
\ee
for any $T'\ge0$. 

ii)  By Lemma \ref{average-lm}, if $\eps\le \eps_1=\eps_1(\frac{\delta_1}{2}, T^*)>0$, then 
\be\label{85}
\sup_{0 \le \tau \le T^*} 
\| \cD\big( a^\eps (\tau)) -   \cD\big( a^0 (\tau;v_0) \big) \|_{L}^* \le     \tfrac{\delta_1}2. 
\ee
In particular, in view of \eqref{84} with $T'=0$, 
\be\label{855}
\| \cD\big( a^\eps (T^*)\big) -  \mu^0 \|_{L}^* <  \delta_1.
\ee

iii) 
By \eqref{855} and 
\eqref{83a} with $\delta= \delta_1$ and with  $T'= nT^*$, $n=1, 2, \dots $ we get inductively  that 
\be\label{856}
\| \cD\big( a^\eps (nT^*)\big) -  \mu^0 \|_{L}^* \le  \delta_1 \quad \forall\, n\in\Nn,
\ee
if $\eps\le\eps_2=\eps_2(\delta_1)$. 

iv) Now by \eqref{856} and \eqref{83b} with $\delta= \delta_1$, 
 for any $n\in \Nn$ and $0\le \theta\le T^*$, 
\be\label{857}
\| \cD\big( a^\eps (nT^* +\theta)) -  \mu^0 \|_{L}^* \le\delta_1/2 + \sup_{\theta\ge0} g_{M_*}(\theta, \delta_1),
\ee
if $\eps\le\eps_2(\delta_1)$.

v) Finally, if $\eps \le \eps_\# (\delta_1) =  \min\big\{\eps_1, 
 \eps_2\big\}$, then by \eqref{85} if $\tau\le T^*$ and by \eqref{84}+\eqref{857} if $\tau\ge T^*$
 we have that  
$$
\| \cD\big( a^\eps (\tau)) -   \cD\big( a^0 (\tau;v_0) \big) \|_{L}^* \le \delta_1 + \sup_{\theta\ge0} g_{M_*}(\theta, \delta_1)\qquad \forall\, \tau\ge0.
$$
By Lemma \ref{p_suff_cond}, for  $M_*$ fixed the function   $g_{M_*}(\theta, d)$ is uniformly continuous in $d$ and 
vanishes at $d=0$. So there exists $\delta^* =\delta^*(\delta)$, which we may assume to be $\le \delta/4$, such that if $\delta_1 =\delta^*$, then 
$g_{M_*}(\theta, \delta_1) \le \delta/2$  for every $\theta\geqslant0$. Then by the estimate above,
$$
\| \cD\big( a^\eps (\tau)) -   \cD\big( a^0 (\tau;v_0) \big) \|_{L}^* \le
 \delta \quad\text{if} \quad \eps \le \eps_*(\delta) :=   \eps_\# (\delta^*\big(\delta)\big)>0, 
$$
for every positive $\delta$. This proves the theorem's assertion. 
\end{proof}

\subsection{Proof of Lemma \ref{p_suff_cond} }\label{ss_proof_prop}  
In this proof we write solutions of effective equation \eqref{eq-ef1} as $a(\tau)$ (rather than $a^0(\tau)$). 
 
 i) At this step, for any non-random $v^1, v^2\in  \bar B_M(h^{s_*})$ we denote  $a_j^\omega(\tau) = a^\omega(\tau, v^j)$, $j=1,2$,
 and   examine the distance 
 $ \| \cD (a_1(\tau))-  \cD (a_2(\tau))\|_L^*$ 
  as a function of $\tau$ and $|v^1-v^2|_{s_*}$. We assume that $| v^1-v^2|_{s_*}\le \bar d $ for some $\bar d\ge0$ and set 
 $w^\omega(\tau) = a_1^\omega(\tau) -a_2^\omega(\tau)$. 
  We have 
 $$
 \dot w^\omega  
   = -  A_V w^\omega+R(a_1^\omega) - R(a_2^\omega).
 $$
   So by Duhamel's principle, 
 $$
w^\omega(\tau)=e^{-  A_V\tau}w(0)+\int_0^\tau e^{-  A_V(\tau-t)}(R(a^\omega_1(t))-R(a^\omega_2(t)))dt.
 $$
 Here $| w(0)|_{s_*} \le \bar d$, and by  (1) of Assumption \ref{assume1} 
 $$
 |R(a_1^\omega(\tau)) - R(a^\omega_2(\tau))|_{s_*-1} \le C |w^\omega(\tau)|_{s_*}\, X^\omega(\tau), \quad X^\omega(\tau) =
 1 + |a^\omega_1(\tau)|_{s_*}^{ \bar m}  + |a^\omega_2(\tau)|_{s_*}^{ \bar m}.
 $$
 For $\theta>0$ the norm of the operator
 $
 e^{-  A_V\theta}: H^{s_*-1} \to H^{s_*}
 $
 is bounded by $\chi(\theta)$, where 
 $$
 \chi(\theta)
 =\begin{cases}
  C\theta^{-1/2}, & 0<\theta \le1,\\
  C e^{-c \theta},&\theta > 1,
  \end{cases}
  $$
 with  some $C,c >0$. So if
 $
 \max_{0\le t \le \tau} X^\omega(t) \le K, 
 $
then  
 $$
 | w^\om(\tau)|_{s_*} \le \bar de^{-\lambda_1 \tau}  +KC \int_0^\tau \chi(\tau-l) | w^\om(l)|_{s_*}dl.
 $$
 Applying Gronwall's lemma we derive from here that 
 \be\label{new_est}
 | w^\om(\tau)|_{s_*} \le \bar d (1+ CK e^{c_1 K} ), \quad \forall\, \tau\ge 0.
 \ee 
 for some positive constants $C,c_1$. 
 
  Denote  $ Y(T) = \sup_{0\le t\le T } |X^\omega(t)|$.  By \eqref{up-1}, $$\mathbf{E}Y(T)\leqslant 2(1+C'_{s_*}(M))(T+1).$$  
   For $K>0$ let  $ \Omega_K(T)$  be the event  $ \{ Y(T) \ge K\}$. Then 
 $
 \PP (\Omega_K(T)) \le 2(1+C'_{s_*}(M)) (T+1) K^{-1},
 $
 and
 $| a_1(\tau ) - a_2(\tau ) |_{s_*} = | w^\om(\tau)|_{s_*}$ satisfies \eqref{new_est}   for $\omega\notin \Omega_K(T)$.
  From here  we see that if $f\in C_b(h^{s_*})$ is such that $|f|_{C_b(h^{s_*})}\le1$ and Lip$\,f\le1$, then  
 \be\label{split2}
 \begin{split} 
 \EE \big(  f(a_1(\tau)) &-  f(a_2(\tau) ) \big) \le 2 \PP(\Omega_K(\tau )) + \bar d  (1+ CK e^{c_1K})
 \\
  &
 \leqslant 4(1+C'_{s_*}(M)) (\tau +1) K^{-1} + \bar d  (1+ CK e^{c_1K}) 
 \quad \forall\, K>0.
 \end{split} 
 \ee
 Let us denote by $g^1_M(\tau ,\bar d)$ the function in the r.h.s. above  with $K= \ln \ln\big( \bar d^{-1} \vee 3 \big)$.  
 This is a continuous function of $(M,\tau , \bar d) \in \R_+^3$, 
 vanishing when $\bar d=0$. Due to \eqref{mix-b1}
 and  \eqref{split2}, 
 \be\label{split3}
 \begin{split}
 \| \cD(a(\tau;v^1)) - &\cD(a(\tau;v^2)) \|_L^* =
 \| \cD(a_1(\tau)) - \cD(a_2(\tau)) \|_L^* \\
 &\le (2{\frak g}_M(\tau) )\wedge g^1_M(\tau,\bar d) \wedge 2 =: g_M^2(\tau,\bar d) ,
  \end{split}
 \ee
 if $ v_1, v_2 \in \bar B_M(h^{s_*})$ and $  |v^1-v^2|_{s_*} \le \bar d$, for any $M, \bar d>0$.  
 The function $g^2$ is continuous in the variables $(M,\tau, \bar d)$, 
 vanishes with $\bar d$ and goes to zero when $\tau \to\infty$ since $ \frak g_M(\tau)$ does. 
   \smallskip
 
 ii) At this step we consider a solution $a^0(\tau;\mu) =: a(\tau;\mu)$ of \eqref{eq-ef1} as in the lemma and examine the l.h.s. of \eqref{mix-b2} as a function of $\tau$. 
  For any $K>0$ consider   the conditional probabilities 
$
\mu_K =  \mu(\cdot  \mid  \bar B_K(h^{s_*}))$ and \mbox{$ \bar\mu_K =\mu(\cdot \mid h^{s_*}\setminus \bar B_K(h^{s_*}))$}.
  Then 
  $
  \mu = A_K \mu_K + \bar A_K \bar\mu_K, 
  $
   where  $A_K= \mu(\bar B_K(h^{s_*}))$ and $ \bar A_K = \PP\{ |a(0)|_{s_*} >K\}
 \leqslant M/K^{2\bar m}$  as $\EE |a(0)|_{s_*}^{2\bar m} \le M$. So 
 \begin{equation}\label{split4}
 \cD(a(\tau,\mu))=A_K\cD (a(\tau;\mu_K)) +\bar A_K\cD (a(\tau;\bar\mu_K)).
 \end{equation} 
 In view of  Assumption \ref{assume2},
   \[
 \begin{split}
  \|\cD (a(\tau;\mu_K)) -\mu^0\|_L^* & =   \| \int\big[ \cD (a(\tau; v)) \big] \mu_K(dv)   -\mu^0\|_L^*\\
  &  \le \int \| \cD (a(\tau; v)) -\mu^0\|_L^* \mu_K(dv) \le
   \frak g_K(\tau).
 \end{split}
 \]
 Therefore due to \eqref{split4},
  \[ \begin{split}
 &\|\cD(a(\tau,\mu))-\mu^0\|_L^* \leqslant 
  A_K   \|\cD(a(\tau,\mu_K))-\mu^0\|_L^* + \bar A_K   \|\cD(a(\tau,\bar\mu_K))-\mu^0\|_L^* \\
 & \le     \|\cD(a(\tau,\mu_K))-\mu^0\|_L^* + 2\bar A_K   \le
  \frak g_K(\tau) + 2
  \frac{M}{K^{2\bar m}} \quad\text{for $ K>0$ and  $\tau\ge 0 $. }
  \end{split}
  \]  
 Let $K_1(\tau)>0$ be a continuous non-decreasing function such that $K_1(\tau)\to\infty$ and $\frak g_{K_1(\tau)}(\tau)\to0$ as $\tau\to\infty$
 (it exists since $\frak g_K(\tau)$ is a continuous function of $(K,\tau)$, going to 0 as $\tau\to\infty$  for each fixed $K$). 
   Then choosing in the estimate above $K=K_1$ we get 
 \begin{equation}\label{y2}
 \|\mathcal{D}(a(\tau;\mu))-\mu^0\|_L^*\leqslant  \frak g_{K_1(\tau)}(\tau)+   \frac{2M}{K_1(\tau)^{2\bar m}}=:
 {\hat {\frak g}}_{M}(\tau).
 \end{equation}
 Obviously $ {\hat{\frak g}}_{M}(\tau)\ge 0$ is a continuous function on $\R_+^2$ ,  converging to $0$ as $\tau\to\infty$.

     \smallskip
 iii)   
 Now we examine  the l.h.s. of \eqref{mix-b2} as a function of $\tau$ and $d$.   Recall that the Kantorovich distance between measures 
 $\nu_1, \nu_2$ on $h^{s_*}$ is 
 $$
 \| \nu_1 -\nu_2\|_K = \sup_{\text{Lip}\, f\le1} \lan f, \nu_1\ran - \lan f,\nu_2\ran \le \infty. 
 $$
 Obviously $  \| \nu_1 -\nu_2\|_L^* \le  \| \nu_1 -\nu_2\|_K$. Since the $2\bar m$-th moments of $\mu$ and $\mu^0$ are bounded by $M\vee C_{0}$ by 
  \eqref{mu0-b1}  and  the assumption on  $\mu$  and since $\| \mu-\mu^0\|_L^* \le d$, then
 \be\label{x7}
 \| \mu- \mu^0\|_K \le \tilde{C} (M\vee C_{0}) ^{\gamma_1} d^{\gamma_2}:=\tilde d, \quad \gamma_1 = \tfrac{1}{2\bar m} , \; 
 \gamma_2 = \tfrac{2\bar m-1}{2\bar m} ;
 \ee
   see  \cite[Section 11.4]{BK} and    \cite[Chapter 7]{Vil}.  By the Kantorovich--Rubinstein theorem (see \cite{Vil, BK})
  there exist r.v.'s $\xi$ and $\xi_0$, defined 
 on a new probability space $(\Omega', \cF', \PP')$,  such that $\cD( \xi )= \mu$,  $\cD( \xi_0) = \mu^0$ and 
 \be\label{x77}
 \EE\,  |\xi_1 -\xi_0|_{s_*} = \| \mu- \mu^0\|_K.
 \ee
 Then using \eqref{split3} and denoting by $ a_{st}(\tau)$ a stationary solution of \eqref{eq-ef1},  $\cD (a_{st}(\tau))\equiv \mu^0$, 
   we have:
 \[
 \begin{split} 
 \| \cD (a(\tau)) - \mu^0\|_L^* =  \| \cD a(\tau; \mu^0) - \cD (a_{st}(\tau))  \|_L^* \le
 \EE^{\om'} \| \cD (a(\tau; \xi^{\om'})) - \cD (a(\tau; \xi_0^{\om'} ))  \|_L^* \\
 \le  \EE^{\om'} g^2_{\bar M}(\tau, |\xi^{\om'} -\xi_0^{\om'}|_{s_*}), \qquad {\bar M}={\bar M}^{\om'}=|\xi^{\om'}|_{s_*} \vee  |\xi_0^{\om'}|_{s_*} .
  \end{split} 
 \]
As $\EE ^{\om'} {\bar M}^{2\bar m} \le 2 (M\vee C_{0})$, then for any $K>0$, 
$$
\PP^{\om'} (Q'_K) \le  2K^{-2\bar m} (M\vee C_{0}), \qquad Q'_K =\{ \bar M\ge K\} \subset \Omega'.
$$
Since $g^2 \le 2$ and for $\om' \notin Q'_K$ we have $ |\xi^{\om'}|_{s_*} , |\xi_0^{\om'}|_{s_*} \le K$,  then 
 $$
 \| \cD (a(\tau)) - \mu^0\|_L^*  \le 4K^{-2\bar m} (M\vee C_{0}) + \EE ^{\om'} g_K^2(\tau,  |\xi^{\om'} -\xi_0^{\om'}|_{s_*}) \quad \forall\, K>0.
 $$

 For an $r>0$    let us denote  $\Omega'_r =\{  |\xi^{\om'} -\xi_0^{\om'}|_{s_*} \ge r\}$. Then by \eqref{x77} and 
 \eqref{x7},  $\PP^{\om'} \Omega'_r\le \tilde d r^{-1}$. So
 \be\label{x8}
 \| \cD (a(\tau)) - \mu^0\|_L^*  \le 4K^{-2\bar m} (M\vee C_{0}) +  2\tilde d r^{-1} + g^2_K(\tau,r), \qquad \text{  for any $K, r>0$. }
  \ee
  Let $g_0(l)$ be a positive continuous function on $\mathbb{R}_+$ such that $g_0(l)\to\infty$ as
   $l\to+\infty$ in such a way that 
    $|C'_{s_*}\big(g_0(l)\big)(\ln\ln l)^{-1/2}|\leqslant 2C'_{s_*}(0)$ for $l\geqslant3$.
  With $r=\tilde d^{1/2}$ and $K=g_0(r^{-1})$, we denote the  r.h.s of \eqref{x8} 
  as $g^3_M(\tau,r)$ (so we substitute in \eqref{x8} 
   $\tilde d=r^2$ and $K=g_0(r^{-1})$). 
 By \eqref{x8} and  the definition of $g^2$ (see  \eqref{split3}), we have 
 \[ \begin{split}
 & \| \cD (a(\tau)) - \mu^0\|_L^* \le
 g_{M}^3(\tau,r) 
 \leqslant 4(g_0(r^{-1}))^{-2\bar m}(M\vee C_{0})+2r \\
 &+4(\tau+1)\big(1+C'_{s_*}\big(g_0(r^{-1})\big)\big)
 \big(\ln\ln(r^{-1}\vee 3)\big)^{-1}
 +rC\ln\ln(r^{-1}\vee3)\exp\big(c_1\ln\ln (r^{-1}\vee3)\big).
 \end{split}
 \]
 As $r\to0$ the  second  and  fourth terms converge to  zero. By the choice of $g_0$, the first term clearly converges to zero with $r$,   so does the third term, which  
  is   $\leqslant 8(1+\tau)(1+C'_{s_*}(0))(\ln\ln(r^{-1}))^{-1/2}$ for $r\leqslant\frac{1}{3}$.   
Hence  $g_M^3(\tau,r)$ defines a continuous function on $\mathbb{R}_+^3$, vanishing with $r$.  Using 
\eqref{x7} let us write $r=\tilde d^{1/2}$ as $r=R_{M}(d)$, where $R$ is a continuous function $\R_+^2 \to \R_+$,  non-decreasing in $d$ 
and vanishing with $d$. Setting 
$
g^4_M(\tau, d) = g^3_M(\tau, R_{M}(d \wedge 2))
$
and using that $\| \mu -\mu^0\|_L^* \le2$, we get from the above that 
$$
 \| \cD (a(\tau)) - \mu^0\|_L^* \le g^4_M(\tau, \| \mu -\mu^0\|_L^* ). 
$$
Finally, evoking \eqref{y2} we arrive at \eqref{mix-b2} with $g_M=g_M^5$, where 
$$
g^5_M(\tau, d) = g^4_M(\tau, d) \wedge \hat{\frak{ g}}_{M}(\tau) \wedge 2, \quad 0\le d \le2. 
$$
The function $g^5$ is continuous, vanishes with $d$ and converges to zero as $\tau\to\infty$. For any fixed $M>0$  this convergence 
is uniform in $d$ due to the term  $\hat{\frak{ g}}_{M}(\tau)$. So for a fixed $M>0$ the function $(\tau, d) \mapsto g^5_M(\tau, d) $ extends to a 
continuous function on the compact set $[0, \infty] \times [0,2]$ (where it vanishes when $\tau=\infty$). 
Thus $g^5_M$ is uniformly continuous in $d$, and the lemma is proved.

\section{NLW: the setting and  result}\label{s_nlw}

In this section we briefly discuss  stochastic nonlinear wave (NLW) equations. Following 
 \cite{m2014, mn2018} we consider the following equations on a smooth bounded domain $\frak{D}\subset \mathbb{R}^3$:
\begin{equation}\label{nlw1}
\partial_t^2u+\gamma \partial_t u-\Delta u=-\gamma f(u)+\gamma h(x)+\sqrt{\gamma}\,\eta(t,x), \quad x\in \frak D, 
\end{equation}
supplemented with the Dirichlet boundary condition on $\p\frak D$. Here 
 $\gamma\in(0,1]$ is a small parameter, $h(x)$ is a function in $H^1_0(\frak D;\mathbb{R})$ and the nonlinearity $f$ is $C^2$-smooth. 
 The random  force $\eta(t,x)$ is a white noise of time of the form 
\[\eta(t,x)=\frac{\partial}{\partial t}\sum_{j=1}^\infty b_j\beta^R_j(t)\frak e_j(x).\] 
Here $\{\beta^R_j(t),j\geqslant1\}$ is  a sequence of independent standard real Brownian motions, $\{\frak e_j(x),j\geqslant1\}$ is an orthonormal basis in $L^2(\frak D;\mathbb{R})$ composed of  eigenfunctions of the  Laplacian operator; that is  $-\Delta\frak e_j=\lambda_j\frak e_j$ with $0<\lambda_1\leqslant\lambda_2\leqslant\dots$.  The set $\{b_j,j\geqslant1\}$ is a sequence of positive real numbers, satisfying 
\be\label{B1}
\mathcal{B}:=\sum_{j\ge1} \lambda_j b_j^2<+\infty.
\ee
The nonlinear term $f$   meets the following growth condition,
\begin{equation}\label{f-g1}
|f''(u)|\leqslant C(|u|^{\rho-1}+1),\; \;\;u\in\mathbb{R},
\end{equation}
where $C$ and $\rho<2$ are positive constants,  as well as  the dissipativity conditions
\begin{equation}\label{f-d1}
F(u)\geqslant -\kappa u^2-C,\;\;\; u\in\mathbb{R},
\end{equation}
\begin{equation}\label{f-d2}
f(u)u-F(u)\geqslant -\kappa u^2-C,\; u\in\mathbb{R},
\end{equation}
where $F$ is the primitive of $f$ and  $\kappa$ is a positive constant.
\begin{lemma} \label{smoothing1}(\cite[Lemma 4.5]{m2014}). Under the condition \eqref{f-g1} we have 
\[\|f(u)-f(v)\|_0^2\leqslant C(\|u\|^{2\rho}_{1-s_\rho}+\|v\|^{2\rho}_{1-s_\rho}+1)\|u-v\|_{1-s_\rho}^2,\quad s_\rho=\tfrac{2-\rho}{2(\rho+1)}.
\]
\end{lemma}

Let us denote $\partial_tu=Lv$, where $L=(-\Delta)^{1/2}$, and set $\xi=u+iv$. Then in terms of $\xi$  equation \eqref{nlw1} reads
\[\partial_t\xi=iL\xi-\gamma i\Im \xi-\gamma iL^{-1}f(\Re\xi)-\gamma iL^{-1}\eta(t,x).\]
Introducing the slow time $\tau=\gamma t$, we have 
\begin{equation}\label{nlw2}
\partial_\tau\xi=\gamma^{-1}iL\xi-i\Im\xi-iL^{-1}f(\Re\xi)-i\tilde\eta(\tau,x),
\end{equation}
where $\tilde\eta(\tau,x)=\sum_{j=1}^\infty \tilde b_j  \dot \beta^R_j(\tau)\frak e_j(x)$ with $\tilde b_j= b_j\lambda_j^{-1/2}$, $j\geqslant1$. 

The NLW equation  \eqref{nlw1} (without the $\gamma$-factor in the r.h.s) was studied in \cite{m2014,mn2018}, where the global well-posedness and
estimates for the  norms   of  solutions  are established in the Sobolev space $H^s\times H^{s-1}= \{(u, \dot u)\}$,\  $s\in [1,2-\frac{\rho}{2}]$.  A simple analysis of the proof of  \cite[Proposition 3.4]{m2014} or \cite[Proposition 5.4]{mn2018} implies 
 the following statement on the global well-posedness  of eq. \eqref{nlw2}:
 
\begin{theorem}\label{nlw-apriori} Assume  conditions \eqref{f-g1}, \eqref{f-d1} and \eqref{f-d2}. Then for any $s\in[1,2-\frac{\rho}{2})$ and $\xi_0\in H^s(\frak D;\mathbb{C})$,  equation \eqref{nlw2} has a unique strong solution $\xi^\gamma(\tau;\xi_0)$, equal to $\xi_0$ at $\tau=0$ and defined for  $\tau\geqslant0$.
Uniformly in $\gamma\in(0,1]$ this solution satisfies 
\begin{equation}\label{nlw-ap1}\mathbf{E}\sup_{\tau\in[\theta,\theta+1]}\|\xi^\gamma (\tau;\xi_0)\|_s^{2m}\leqslant C_s(m,\|\xi_0\|_s,\mathcal{B}),\;\;\; \forall \theta\geqslant0,\;m\in\mathbb{N},\end{equation}
where $C_s$ is a continuous function, non-decreasing in all its arguments. 
\end{theorem}

Now let us write eq. \eqref{nlw2} in terms of  Fourier coefficients with respect to the basis 
$\{\frak e_j(x)\}$: $\xi=\sum_{j\geqslant1}v_j\frak e_j$. As in Section \ref{s_2}, we use $h^s$ to denote the space of sequences  of complex 
 Fourier coefficients and denote by   $\Psi$ 
 the map  $H^s(\frak D;\mathbb{C})\to h^s$, $\xi\mapsto v:=(v_j,j\geqslant1)$. Then we have
\begin{equation}\dot v_k=\gamma^{-1}i\lambda_k^{1/2}v_k-i\Im v_k+i\hat P_k(v)+i\tilde b_k\beta_k^R(\tau),\quad  k\in\mathbb{N},
\end{equation}
where $\hat P(v)=(\hat P_k(v),k\in \mathbb{N})=\Psi(L^{-1}f(\Re\xi))$ with $\xi=\Psi^{-1}v$. Passing to  the interaction representation
\[
a_k(\tau)=e^{-i\gamma^{-1}\lambda_k^{1/2}\tau}v_k(\tau),\quad k\geqslant1,
\]
 we obtain the following system of  equations for the $a$-variables
\begin{equation}\label{nlw-a1}\dot a_k(\tau)=ie^{-i\gamma^{-1}\lambda_k^{1/2}\tau}\Big(-\Im\big(e^{i\gamma^{-1}\lambda_k^{1/2}\tau} a_k\big)+\hat P_k(\Phi_{\gamma^{-1}\hat\Lambda\tau}a)+ \tilde b_k\dot\beta_k^R\Big),\;\;  \; k \in\mathbb{N},
\end{equation}
where $\hat\Lambda=(\lambda_k^{1/2},k\geqslant1)$ and operator $\Phi$ is defined  in \eqref{Phi-d}. Let us calculate an effective equation for \eqref{nlw-a1},
following \cite{HKM}. 

i) To calculate the effective diffusion we decomplexify  components $a_k = a_k^R +i a_k^I$  of the $a$-vector 
as $(a_k^R +a_k^I) \in \R^2$ and write the dispersion matrix of the 
equation as a block-diagonal real matrix with the blocks 
$
\tilde b_k 
\begin{pmatrix}
\cos\phi_k & -\sin\phi_k \\
0 & 0
\end{pmatrix},
$
where
$ \phi_k = \phi^k(\tau) =
 \gamma^{-1} \lambda_k^{1/2} \tau$ for $ k=1,2,\dots .
$
So the diffusion matrix of eq.~\eqref{nlw-a1} -- let us call it 
 $A(\tau)$ --  is formed by  blocks 
$
\tilde b_k^2 \begin{pmatrix}
\cos^2\phi_k & - \cos\phi_k\sin\phi_k \\
- \cos\phi_k\sin\phi_k & \sin^2\phi_k
\end{pmatrix}$. 
Consider the limit
$
\cA(\tau) = \lim_{\gamma\to0} \frac1{\tau} \int_0^\tau A(l)\,dl. 
$
This is a $\tau$--independent block-diagonal real matrix with the blocks $\tfrac12 \tilde b_k^2\, \text{id}$, and this 
 is the diffusion matrix of the effective equation. Corresponding 
dispersion matrix has  blocks $\tfrac1{\sqrt2} \tilde b_k\,\text{id}$. Coming back to the complex coordinates we see that the noise in the effective equation is 
$
\tfrac1{\sqrt2}  \sum \tilde b_k\, \dot{ \tilde\beta}_k(\tau), 
$ 
where $\{\tilde\beta_k(\tau)\}$ are standard independent complex Brownian motions.

ii) The drift in the effective equation is a sum of two terms. The second one is
\[
\tilde{R}({a})=(\tilde R_k({a}),k\geqslant1):=\lim_{T\to\infty}\frac{1}{T}\int_0^{T}\Phi_{-\hat\Lambda t}\hat P(\Phi_{\hat\Lambda t} {a} )dt
\]
(cf.  \eqref{limit-eq}). 
The $k$-th component of  the first term is 
$$
\lim_{T\to\infty}\frac{1}{T}\int_0^Te^{-i\lambda_k^{1/2}t}\big(-i\Im(e^{i\lambda_k^{1/2}t}{a}_k)\big)dt=-i \tfrac{1}2(\Im {a}_k-i\Re {a}_k)=- \tfrac{1}2{a}_k.
$$

iii) So the effective equation for  \eqref{nlw-a1} reads 
\begin{equation}\label{nlw-ef}
\dot{a}_k(\tau)=-\tfrac{1}{2}{a}_k+i\tilde R_k({a})+\tfrac{1}{\sqrt{2}}\tilde b_k d\tilde\beta_k,\qquad k\in\mathbb{N}.
\end{equation}
\smallskip

By Lemma \ref{smoothing1} and \cite[Lemma 4]{HKM}, we have 
\begin{equation}\label{smoothing-ef}
|\tilde R(a_1)-\tilde R(a_2)|_1\leqslant C\big(|a_1|_{1-s_\rho}^\rho\wedge|a_2|_{1-s_\rho}^\rho+1\big)|a_1-a_2|_{1-s_\rho},\;\; \;a_1,a_2\in h^1.
\end{equation}
In particular, the mapping
 ${a}\mapsto\tilde R({a})$ belongs to $\text{Lip}_{\rho+1}(h^s,h^s)$, $s\in[s_\rho,1]$. 
Therefore  effective equation \eqref{nlw-ef} is at least locally well-posed in $h^s$, $s\in[s_\rho,1]$.
Let us denote $a^0(\tau;a_0)$ its solution with an initial condition $a^0(0;a_0)=a_0$. Similarly let
 $a^\gamma (\tau;v_0)$ be a solution  of eq. \eqref{nlw-a1}, equal  $a_0$ at $\tau=0$. 
 Then using \eqref{nlw-ap1} and arguing as in  \cite{HKM} we get
 
 \begin{theorem}\label{nlw-average1} 
 Under the assumptions of Theorem \ref{nlw-apriori} we have:
  
  i) for any ${a}_0\in h^{1}$, eq. \eqref{nlw-ef} has a unique strong solution $a^0(\tau;{a}_0)$, $\tau\geqslant0$.
   It belongs to $C([0,\infty), h^1)$ a.s., and for any $\theta\geqslant0$ satisfies 
  \begin{equation}\label{nlw-ef-ap}
  \mathbf\sup_{\tau\in[\theta,\theta+1]}|a^0(\tau;{a}_0)|_1^{2m}\leqslant C(m,|{a}|_1, \mathcal{B}),\; \;\; \forall m\in\mathbb{N},
  \end{equation}   
  where $C$ is a continuous function, increasing in all its arguments.
  
  ii)  For any $s>1$, ${a}_0\in h^s$ and  $T>0$ we have 
  \be\label{nlw-conv}
  \cD \big(a^\gamma (\tau; {a}_0)\! \mid_{[0,T]} \!\big)  \strela \cD \big(a^0 (\tau; {a}_0)\! \mid_{[0,T]} \!\big) \quad \text{in}\;\;\mathcal{P}\big(C([0,T]); h^1)\big)\;\;
  \text{as} \; \;\gamma \to0. 
  \ee
  Moreover, for any $M>0$  the above convergence is uniform for ${a}_0\in \bar B_M(h^s)$. 
   \end{theorem}  
   

Repeating the argument in Section \ref{s_3} we get the following  analogy to Theorem~\ref{thm-uniform}:

\begin{theorem}\label{thm-nlw}Assume in addition to the assumption in Theorem \ref{nlw-average1} that the effective equation \eqref{nlw-ef} 
is mixing in the space $h^1$  and that   the requirement  as in  Assumption~\ref{assume2} is met. 
Then for any $s>1$ and $v_0\in h^s$ we have the following convergence
\[\lim_{\gamma\to0}\sup_{\tau\geqslant0}\|\cD(a^\gamma (\tau;v_0))-\cD(a^0(\tau;v_0))\|_{L,h^1}^*=0.
\]
 Moreover, for any $M>0$ the above convergence is uniform for $v_0\in \bar B_{M}(h^{s})$. 
\end{theorem}

For the original equation \eqref{nlw1} we introduce the vector of action variables for  its solution $\frak u=(u, u_t)$ as $I(\frak u)=(I_1(\frak u),\dots)$,
 where $I_k(\frak u)=\frac{1}{2}|\langle u,\frak e_k\rangle|^2+\frac{1}{2}\lambda_k^{-1}|\langle  u_t, \frak e_k\rangle|^2$, $k=1,\dots$. 
  Then we have the following corollary in analogy to Corollary \ref{coroll-cgl}:
  
\begin{corollary}Under the assumption of Theorem \ref{thm-nlw}, for any $s>1$ and  $\frak u_0=[u_1,u_2]\in H^s\times H^{s-1}$, the action-vector of a
 solution $\frak{u}^\gamma(t;\frak u_0)$ for equation \eqref{nlw1}, equal  $\frak u_0$ at $t=0$, satisfies 
\[
  \lim_{\gamma\to0}\,
  \sup_{\tau\geqslant0}\|\cD\big(I( \frak u^{\gamma}(\tau\gamma^{-1};\frak u_0)) \big) -\cD\big((I( a^{0}(\tau;v_0))\big) \|_{L,h_I^1}^* =0,
  \]
  where $v_0=\Psi(u_1+iL^{-1}u_2)$.
\end{corollary}

The result of Theorem \ref{thm-nlw} is conditional since it  requires that eq.~\eqref{nlw-a1} is mixing. It is not our goal in this short section to check the 
mixing property. But we mention that since the effective equation \eqref{nlw-ef} is similar to the NLW equation \eqref{nlw1}, written in the form
\eqref{nlw2}, and since by \eqref{nlw-conv} solutions for  \eqref{nlw-ef} inherit the estimates on solutions for \eqref{nlw1}, obtained in 
\cite{m2014, mn2018}, then most likely the proof of the exponential mixing in eq.~\eqref{nlw1}, given in \cite{m2014} (and based on an abstract theorem from \cite{KS}), 
 applies to establish the 
exponential mixing for eq.~\eqref{nlw-ef}  and thus verify for it the analogy of Assumption~\ref{assume2}, required in Theorem~\ref{thm-nlw}.  

\section*{Acknowledgment}
We are thankful to Armen Shirikyan for discussion. 
GH was supported by National Natural Science Foundation of China (Significant project No.11790273).
\bibliography{reference}{}
\bibliographystyle{plain}

\end{document}